\documentclass[11pt, oneside]{amsart}
\newcommand{\bC}{\mathbb{C}}

\newcommand{\bR}{\mathbb{R}}

\newcommand{\diff}{\normalfont\text{Diff}}
\newcommand{\symp}{\normalfont\text{Symp}}

\newtheorem{theorem*}[intro]{Theorem}
\newtheorem{corollary*}[intro]{Corollary}

\usepackage[utf8]{inputenc}
\usepackage[english]{babel}
\usepackage{amsmath}
\usepackage{amssymb}
\usepackage{mathabx}
\usepackage{amsfonts}
\usepackage{amsthm}
\usepackage{enumerate}
\usepackage[dvipsnames]{xcolor}
\usepackage{tikz-cd}
\usepackage{subfiles}
\usepackage[scr]{rsfso}
\usepackage{extarrows}
\usepackage{hyperref}
\usepackage{appendix}
\usepackage{bm}
\hypersetup{
    colorlinks=true,
    linkcolor=Blue,
    filecolor=red,      
    urlcolor=teal,
    citecolor=teal,
}
\newtheorem{theorem}{Theorem}[section]
\newtheorem{lemma}[theorem]{Lemma}

\newtheorem{question}[theorem]{Question}
\newtheorem{cor}[theorem]{Corollary}

\theoremstyle{definition}

\theoremstyle{remark}
\newtheorem{remark}[theorem]{Remark}

\newtheorem*{notation*}{Notation}

\numberwithin{equation}{subsection}

\makeatletter
\@addtoreset{equation}{section}
\@addtoreset{theorem}{section}
\makeatother

\newcommand{\Addresses}{{% additional braces for segregating \footnotesize
  \bigskip
  \footnotesize

  \textsc{Luya Wang, Institute for Advanced Study}\par\nopagebreak
  \textit{E-mail address}: \texttt{luyawang@ias.edu}
}}

\begin{document}

\title{A note on a nontrivial connected component of the space of symplectic structures}
\author{Luya Wang}
\begin{abstract}
    This note provides the first example of a nontrivial connected component of the space of symplectic structures standard at infinity in dimension four.
 \end{abstract}

\maketitle

\section{Introduction}

Given a symplectic manifold $(X, \omega)$, one may consider the space of symplectic structures representing the same cohomology class, denoted as $\mathcal{S}_{[\omega]}$. Two symplectic forms are \emph{isotopic} if there exists a path of cohomologous symplectic forms connecting them. In the case when $X$ is open, we consider only the symplectic structures that are fixed at infinity and relative cohomology. In \cite[\textsection 13.1]{MS17}, \cite[\textsection 9.4]{MS12} and \cite{Sa13}, the following question was asked.\\

\noindent\textbf{Uniqueness question}: \textit{Is the space $\mathcal{S}_{[\omega]}$ connected?} \\

This question was answered in the negative in dimension eight and six in \cite{McDuff_infinite_examples} and \cite[Theorem 9.7.4]{MS12} for closed manifolds. In dimension four, the above uniqueness question is still open for closed manifolds. We give a negative answer to the open case: 

\begin{theorem*}\label{thm:nontrivial}
There exist infinitely many symplectic forms on $S^1 \times D^3$ standard at infinity, such that they are not pairwise isotopic through symplectic forms standard at infinity.
\end{theorem*}

In particular, the space $\mathcal{S}_{0}$ is not connected for $S^1 \times D^3$. The proof is essentially based on either \cite{budney2021knotted} or \cite{watanabe2023thetagraph} which studies the compactly supported diffeomorphisms non-isotopic to the identity, and a result of \cite{Evans} on the compactly supported symplectic mapping class group of $S^1 \times D^3$, itself building on \cite{Gr85}.

\section{Proof of Theorem 1}

Extending the work of Gromov on the group of compactly supported symplectomorphisms $\symp_{c}(D^4, \omega_{0})$ \cite{Gr85}, Evans showed the following. The standard symplectic structure on $S^1 \times D^3 \cong \bC^* \times \bC$ is given by the product symplectic form.

\begin{theorem}\cite[Theorem 6]{Evans}
\label{thm:evans}
    The compactly supported symplectomorphism group $\symp_{c}(S^1 \times D^3, \omega_{\text{std}})$ is weakly contractible.
\end{theorem}

Our result is then an immediate corollary of the works of Budney-Gabai and Watanabe \cite{budney2021knotted,watanabe2023thetagraph}. Among other results, Budney-Gabai and Watanabe proved the following.

\begin{theorem}\cite[Theorem 3.15]{budney2021knotted} \cite[Corollary 1.5]{watanabe2023thetagraph}
\label{thm:watanabe}
    There exist infinitely many elements of $\pi_{0}(\diff_{c}(S^1\times D^3))$.
\end{theorem}

\begin{proof}[Proof of Theorem \ref{thm:nontrivial}]
    Let $f \in \diff_{c}(S^1\times D^3)$ be a nontrivial element from Theorem \ref{thm:watanabe} and $\omega_{1} := f^* \omega_{\text{std}}$. Since $H^2(S^1 \times D^3, S^1\times S^2; \bR)=0$, $\omega_1$ and $\omega_{\text{std}}$ have the same cohomology class. We show that $\omega_{1}$ is not isotopic to $\omega_{\text{std}}$. The pairwise version works exactly the same. Suppose for the contradiction that $\omega_{1}$ is isotopic to $\omega_{\text{std}}$. By Moser isotopy, e.g. \cite[Theorem 3.2.4 and Exercise 3.2.6]{MS17}, we have an isotopy $\phi_t$, such that $\phi_0 =\text{id}$ and $\phi_1^*\omega_{1} = \omega_{\text{std}}$. Therefore, $\phi_1^*f^* \omega_{\text{std}} = \omega_{\text{std}}$. So $f\circ \phi_1$ is a symplectomorphism. Then, by Theorem \ref{thm:evans}, we know that $f\circ \phi_1$ is isotopic to the identity. Contradiction.
\end{proof}

The above arguments can be readily applied to study higher homotopy groups of the space of symplectic structures as well.

\begin{lemma}
\label{lem:higher_homotopy_group}
    Given $k\geq 0$, the $k^\text{th}$-homotopy group of the space of symplectic structures is nontrivial for any $(X, \omega)$ such that $H^2(X, \partial X; \bR) = 0$, $\pi_k(\diff_c(X)) \neq \{*\}$ and $\pi_k(\symp_c(X, \omega)) = \{*\}$. 
\end{lemma}
\begin{proof}
    The proof is similar to that of Theorem \ref{thm:nontrivial}. Consider a nontrivial $k$-parameter family of compactly supported diffeomorphisms $f_{r_1, \dots, r_k}$, where $r_i\in [0,1]$, given by $\pi_k(\diff_c(X)) \neq \{*\}$. Suppose for the contradiction that $f_{r_1, \dots, r_k}^* \omega_{\text{std}}$ is isotopic to $\omega_{\text{std}}$ via symplectic forms that are fixed at infinity. Then by Moser isotopy, we have a $k$-parameter family of compactly supported isotopies $\phi_{t;r_1, \dots, r_k}$ such that $\phi_{1;r_1, \dots, r_k}^* f_{r_1, \dots, t_k}^* \omega_{\text{std}} = \omega_{\text{std}}$ and $\phi_{0;r_1, \dots, r_k} = \text{id}$. Again, now $f_{r_1, \dots, t_k}\circ\phi_{1;r_1, \dots, r_k}$ is a $k$-parameter family of compactly supported symplectomorphisms, which by $\pi_k(\symp_c(X, \omega)) = \{*\}$, must be isotopic to the identity. This contradicts the assumption on $f_{r_1, \dots, r_k}$.
\end{proof}

\begin{cor}[\cite{watanabe2019exoticnontrivialelementsrational} \textsection 1.1(3)]
    The space of symplectic structures on $D^4$ standard at infinity is not contractible.
\end{cor}
\begin{proof}
    By Gromov \cite{Gr85} showing weak contractibility of $\symp_c(D^4, \omega_{\text{std}})$, together with the constructions in \cite[Theorem 1.1]{watanabe2023thetagraph} on nontriviality of certain higher homotopy groups of $\diff_c(D^4)$, we have the conclusion by applying Lemma \ref{lem:higher_homotopy_group}.
\end{proof}
 
\begin{remark}
    As remarked in \cite{watanabe2019exoticnontrivialelementsrational,  Sa13, MS12, Eliashberg, Kirby}, given a symplectic form $\omega$, one can consider the orbits of diffeomorphisms acting on the space $\mathcal{S}_{[\omega]}$. By Moser isotopy, we have that given any two forms in the same connected component of $\mathcal{S}_{[\omega]}$, there is an isotopy connecting them, i.e. the action is transitive on each connected component. Denote the connected component of $\omega$ by $\mathcal{S}_{\omega}$. Now, one needs to restrict to the identity smooth components (or restricting to symplectic forms that are in the orbit under $\diff(X)$ of $\omega$) for $\symp(X, \omega)$ and $\diff(X)$ in order to obtain a fibration 
    $$\symp_0(X, \omega) \to \diff_0(X) \to \mathcal{S}_{\omega}.$$
    To utilize the exotic diffeomorphisms in \cite{watanabe2023thetagraph, watanabe2019exoticnontrivialelementsrational, budney2021knotted}, we therefore prefer to adopt the direct arguments as in the proof of Theorem \ref{thm:nontrivial} and Lemma \ref{lem:higher_homotopy_group}.
\end{remark}

\section{Additional observations and questions}
As evidenced in the proof of Theorem \ref{thm:nontrivial}, the question of finding nontrivial connected components of the space $\mathcal{S}_{[\omega]}$ is deeply connected to finding a diffeomorphism that is not isotopic to the identity. We refer to them as exotic symplectic forms and exotic diffeomorphisms.

The following can be considered as a symplectic version of the Smale conjecture, asked by Eliashberg in \cite[Problem 8]{Eliashberg} and \cite[Problem 4.141(D)]{Kirby}.

\begin{question}
\label{question: Eliashberg}
    Is the space $\mathcal{S}_0$ of all standard-at-infinity symplectic structures on $\bR^4$ connected?
\end{question}

\begin{question}[$\pi_0$ version of $4$-dimensional Smale conjecture]
\label{question: pi_0_smale}
    Is the space $\diff_c(\bR^4)$ of all compactly supported diffeomorphisms on $\bR^4$ connected?
\end{question}

\begin{lemma}
    Question \ref{question: Eliashberg} is equivalent to Question \ref{question: pi_0_smale}.
\end{lemma}

\begin{proof}
    Suppose that we have a symplectic structure $\omega_{\text{exot}}$ on $\bR^4$ that is standard at infinity and not isotopic to $\omega_{\text{std}}$, the standard symplectic form on $\bR^4$. By \cite[\S 0.3.C.]{Gr85}, we know that there exists a compactly supported diffeomorphism $g$ such that $g^*\omega_{\text{exot}} = \omega_{\text{std}}$. However, by assumption there is no compactly supported diffeomorphism isotopic to the identity pulling $\omega_{\text{exot}}$ back to $\omega_{\text{std}}$. Therefore, $g$ is not isotopic to the identity. On the other hand, suppose that $\pi_0(\diff_c(\bR^4)) \neq \{*\}$. Then we can use the same argument as that in the proof of Theorem \ref{thm:nontrivial}, together with Gromov's work on $\pi_0(\diff_c(\bR^4, \omega_{\text{std}})) = \{*\}$ as discussed above, to obtain a nontrivial connected component of $\mathcal{S}_0$ standard-at-infinity on $\bR^4$.
\end{proof}

Therefore, a natural question to ask is whether the forms obtained by attaching the standard Weinstein $2$-handles along the Legendrian knots smoothly isotopic to $S^1 \times \{*\} \subset S^1 \times S^2$ are isotopic or not. Clearly the symplectic form obtained from attaching the standard Weinstein $2$-handle on $(S^1\times D^3, \omega_{\text{std}})$ is the standard symplectic form on $D^4$.

\begin{question}
\label{question: 2handle}
    Are the forms on $D^4$ obtained by attaching the standard Weinstein $2$-handles on $(S^1\times D^3, \omega_i)$ given in Theorem \ref{thm:nontrivial} isotopic to the standard symplectic form?
\end{question}

A negative answer to Question \ref{question: 2handle} would give a negative answer to Question \ref{question: Eliashberg} as well as Question \ref{question: pi_0_smale}. One could hope to construct a isotopy between the symplectic forms obtained by carving out a family of $2$-handles on $D^4$ via composing the Moser isotopy with an isotopy given by \cite{Eliashberg_Polterovich}. The difficulty is to make sure that this family of $2$-handles does not interact with the exotic diffeomorphisms, or interacts in a controlled way.

\begin{remark}
    Since the appearance of this note, a negative answer has been claimed for Question \ref{question: pi_0_smale} in \cite{pi0Smale}. Nevertheless, it would be interesting to find a symplectic proof of Question \ref{question: Eliashberg} directly.
\end{remark}

\subsection*{Acknowledgements} LW thanks Yakov Eliashberg, John Etnyre and Tadayuki Watanabe for generous and enlightening communications. LW is grateful for Kai Cieliebak, Eduardo Fern\'andez, David Gabai, Amanda Hirschi, Ailsa Keating, Ciprian Manolescu, Anubhav Mukherjee, Agniva Roy and Daniel Ruberman for related discussions and interest. LW thanks the referee and acknowledges support from NSF Grant DMS-2303437, IAS Giorgio and Elena Petronio Fellow II Fund, and Simons Collaboration - New Structures in Low-dimensional topology.

\bibliographystyle{amsalpha}
\bibliography{bib}

\providecommand{\bysame}{\leavevmode\hbox to3em{\hrulefill}\thinspace}
\providecommand{\MR}{\relax\ifhmode\unskip\space\fi MR }
\providecommand{\MRhref}[2]{%
  \href{http://www.ams.org/mathscinet-getitem?mr=#1}{#2}
}
\providecommand{\href}[2]{#2}
\begin{thebibliography}{GGH25}

\bibitem[BG21]{budney2021knotted}
Ryan Budney and David Gabai, \emph{Knotted 3-balls in {$S^4$}}, 2021, arXiv:1912.09029.

\bibitem[Eli98]{Eliashberg}
Yasha Eliashberg, \emph{Symplectic topology in the nineties}, vol.~9, 1998, Symplectic geometry, pp.~59--88.

\bibitem[EP96]{Eliashberg_Polterovich}
Yasha Eliashberg and Leonid Polterovich, \emph{Local {L}agrangian {$2$}-knots are trivial}, Ann. of Math. (2) \textbf{144} (1996), no.~1, 61--76. \MR{1405943}

\bibitem[Eva11]{Evans}
Jonathan~David Evans, \emph{Symplectic mapping class groups of some {S}tein and rational surfaces}, J. Symplectic Geom. \textbf{9} (2011), no.~1, 45--82.

\bibitem[GGH25]{pi0Smale}
David Gabai, David~T. Gay, and Daniel Hartman, \emph{Pseudo-isotopy and diffeomorphisms of the 4-sphere {I}: Loops of spheres}, 2025, arXiv:2505.12088.

\bibitem[Gro85]{Gr85}
Misha Gromov, \emph{Pseudo holomorphic curves in symplectic manifolds}, Invent. Math. \textbf{82} (1985), no.~2, 307--347.

\bibitem[Kir97]{Kirby}
\emph{Problems in low-dimensional topology}, Geometric topology ({A}thens, {GA}, 1993) (Rob Kirby, ed.), AMS/IP Stud. Adv. Math., vol. 2.2, Amer. Math. Soc., Providence, RI, 1997, pp.~35--473.

\bibitem[McD87]{McDuff_infinite_examples}
Dusa McDuff, \emph{Examples of symplectic structures}, Invent. Math. \textbf{89} (1987), no.~1, 13--36.

\bibitem[MS12]{MS12}
Dusa McDuff and Dietmar Salamon, \emph{{$J$}-holomorphic curves and symplectic topology}, second ed., American Mathematical Society Colloquium Publications, vol.~52, American Mathematical Society, Providence, RI, 2012.

\bibitem[MS17]{MS17}
\bysame, \emph{Introduction to symplectic topology}, third ed., Oxford Graduate Texts in Mathematics, Oxford University Press, Oxford, 2017.

\bibitem[Sal13]{Sa13}
Dietmar Salamon, \emph{Uniqueness of symplectic structures}, Acta Math. Vietnam. \textbf{38} (2013), no.~1, 123--144.

\bibitem[Wat19]{watanabe2019exoticnontrivialelementsrational}
Tadayuki Watanabe, \emph{Some exotic nontrivial elements of the rational homotopy groups of $\mathrm{Diff}({S}^4)$}, 2019, arXiv:1812.02448.

\bibitem[Wat23]{watanabe2023thetagraph}
\bysame, \emph{Theta-graph and diffeomorphisms of some 4-manifolds}, 2023, arXiv:2005.09545.

\end{thebibliography}

\Addresses

\end{document}